\def\lesscirc{\mathrel{\mathord{<}\!\!\raise
 0.96 pt\hbox{$\scriptstyle\circ$}}}

\def\name#1{\mathchoice
{\setbox0=\hbox{$\displaystyle #1$}
\setbox1=\vtop{\ialign{##\crcr
$\hfil{\displaystyle #1}\hfil$\crcr\noalign{\kern2pt%
\nointerlineskip}$\hfil\mathord{\displaystyle \sim}%
\hfil$\crcr\noalign{\kern3pt\nointerlineskip}}}%
\setbox2=\hbox{$\displaystyle \sim$}%
\box1}%
{\setbox0=\hbox{$\textstyle #1$}
\setbox1=\vtop{\ialign{##\crcr
$\hfil{\textstyle #1}\hfil$\crcr\noalign{\kern2pt%
\nointerlineskip}$\hfil\mathord{\textstyle \sim}%
\hfil$\crcr\noalign{\kern3pt\nointerlineskip}}}%
\setbox2=\hbox{$\textstyle \sim$}%
\wd1=\wd0\dp1=0cm\ifdim\wd2>\wd1 \wd1=\wd2\else\relax\fi
\ht1=\ht0\relax
\box1}{\setbox0=\hbox{$\scriptstyle #1$}
\vtop{\ialign{##\crcr
$\hfil{\scriptstyle #1}\hfil$\crcr\noalign{\kern1.4pt%
\nointerlineskip}$\hfil\mathord{\scriptstyle \sim}%
\hfil$\crcr\noalign{\kern2.1pt\nointerlineskip}}}%
}{\setbox0=\hbox{$\scriptscriptstyle #1$}
\vtop{\ialign{##\crcr
$\hfil{\scriptscriptstyle #1}\hfil$\crcr\noalign{\kern1pt%
\nointerlineskip}$\hfil\mathord{\scriptscriptstyle \sim}%
\hfil$\crcr\noalign{\kern1.5pt\nointerlineskip}}}%
}}

\def\nameb#1{\mathchoice%
{}%
{}%
{\setbox0=\hbox{$\scriptstyle #1$}
\vcenter{\ialign{##\crcr
$\hfil{\scriptstyle #1}\hfil$\crcr\noalign{\kern1.4pt%
\nointerlineskip}$\hfil\mathord{\scriptstyle \sim}%
\hfil$\crcr\noalign{\kern2.1pt\nointerlineskip}}}%
}{\setbox0=\hbox{$\scriptscriptstyle #1$}
\vcenter{\ialign{##\crcr
$\hfil{\scriptscriptstyle #1}\hfil$\crcr\noalign{\kern1pt%
\nointerlineskip}$\hfil\mathord{\scriptscriptstyle \sim}%
\hfil$\crcr\noalign{\kern1.5pt\nointerlineskip}}}%
\setbox2=\hbox{$\scriptscriptstyle \sim$}%
\box1}}

\documentclass[a4paper,11pt,oneside]{amsart}

\usepackage{pdfsync}

\usepackage{amsmath,amsfonts,amssymb,amsthm,amscd,latexsym,euscript,xypic}

\renewcommand{\emptyset}{\varnothing}

\newtheorem{theorem}{Theorem}
\newtheorem{lemma}[theorem]{Lemma}
\newtheorem{claim}[theorem]{Claim}

\newtheorem{corollary}[theorem]{Corollary}
\newtheorem{definition}[theorem]{Definition}

\newcommand{\Pe}{\Bbb P}

\newcommand{\Qu}{\Bbb Q}

\newcommand{\Acal}{\mathcal{A}}

\begin{document}
\title[On partial orderings having precalibre-$\aleph_1$ ...]{On partial orderings having precalibre-$\aleph_1$ and fragments of Martin's axiom}

\author{Joan Bagaria}
\address{ICREA (Instituci\'o Catalana de Recerca i Estudis Avan\c{c}ats) and
Departament de L\`ogica, Hist\`oria i Filosofia de la Ci\`encia, Universitat de Barcelona. 
Montalegre 6,
08001 Barcelona, Catalonia (Spain).}
\email{joan.bagaria@icrea.cat}
\thanks{The research work of the first author was partially supported by the Spanish 
Government under grant MTM2011-25229,
and by the Generalitat de Catalunya (Catalan Government) under grant 2009~SGR~187. The research of the second author was supported by the United States-Israel Binational Science Foundation. Publication 1041 on his list.}
\author{Saharon Shelah}
\address{Einstein Institute of Mathematics, The Hebrew University of Jerusalem. 
Edmond J. Safra Campus, Givat Ram, 
Jerusalem 91904, Israel.}
\email{shelah@huji.ac.il}
\thanks{}

\date{\today }

\begin{abstract}
We define a countable antichain condition (ccc) property for partial orderings,   weaker than  precalibre-$\aleph_1$, and show that Martin's axiom restricted to the class of partial orderings that have the property does not imply Martin's axiom for $\sigma$-linked partial orderings. This answers an old question of the first author about the relative strength of Martin's axiom for $\sigma$-centered partial orderings together with the assertion that every Aronszajn tree is special. We also answer a question of J. Steprans and S. Watson (1988) by showing that, by a forcing that preserves cardinals,  one can destroy   the precalibre-$\aleph_1$ property of a partial ordering while preserving its ccc-ness.

\end{abstract}

\maketitle

 A question asked in \cite{Ba} is if $MA(\sigma\mbox{-centered})$ plus ``Every Aronszajn tree is special" implies $MA(\sigma\mbox{-linked})$. The interest in this question originates in the result of Harrington-Shelah \cite{HS} showing that if $\aleph_1$ is accessible to reals, i.e., there exists a real number $x$ such that the cardinal $\aleph_1$ in the model $L[x]$ is equal to the real $\aleph_1$, then $MA$ implies that there exists a $\Delta^1_3(x)$ set of real numbers that does not have the Baire property. The hypothesis that $\aleph_1$ is accessible to reals is necessary, for if $\aleph_1$ is inaccessible to reals and $MA$ holds, then $\aleph_1$ is actually weakly-compact in $L$ (\cite{HS}), and K. Kunen showed that starting form a weakly compact cardinal one can get a model where $MA$ holds and every projective set of reals has the Baire property. In \cite{Ba}, using Todor\v{c}evi\'c's $\rho$-functions (\cite{To}), it was shown that $MA(\sigma\mbox{-centered})$ plus ``Every Aronszajn tree is special'' is sufficient to produce a $\Delta^1_3 (x)$ of real numbers without the Baire property, assuming $\aleph_1=\aleph_1^{L[x]}$. Thus, it was natural to ask how weak is $MA(\sigma\mbox{-centered})$ plus ``Every Aronszajn tree is special'' as compared to the full  $MA$, and in particular if it implies $MA(\sigma\mbox{-linked})$.
 We answer the question in the negative by showing that, in fact, a fragment of $MA$ that includes $MA(\sigma\mbox{-centered})$, and even $MA(3\mbox{-Knaster})$, and implies  ``Every Aronszajn tree is special'', does not imply $MA(\sigma\mbox{-linked})$. A partial ordering with the precalibre-$\aleph_1$ property plays the key role in the construction of the model.

In the second part of the paper we answer a question of Steprans-Watson \cite{SW}. They ask if it possible to destroy the precalibre-$\aleph_1$ property of a partial ordering, while preserving its ccc-ness, in a forcing extension of the set-theoretic universe $V$ that preserves cardinals.  This is a natural question considering that, as shown in \cite{SW}, on the one hand, assuming $MA$ plus the Covering Lemma, every precalibre-$\aleph_1$ partial ordering has precalibre-$\aleph_1$ in every forcing extension of $V$ that preserves cardinals; and on the other hand the ccc property of a partial ordering having precalibre-$\aleph_1$ can always be destroyed while preserving $\aleph_1$, and consistently even preserving all cardinals.

We answer the Steprans-Watson question positively, and in a very strong sense. Namely, we show that it is consistent, modulo ZFC, that the Continuum Hypothesis holds and there exist a forcing notion $T$  of cardinality $\aleph_1$ that preserves $\aleph_1$ (and therefore it preserves all cardinals, cofinalities, and the cardinal arithmetic), and two precalibre-$\aleph_1$ partial orderings, such that forcing with $T$ preserves their ccc-ness, but it also forces that their product is not ccc and therefore they don't have precalibre-$\aleph_1$.

\section{Preliminaries}

Recall that a partially ordered set (or poset) $\Pe$ is \emph{ccc} if every antichain of $\Pe$ is countable;  it is \emph{productive-ccc} if the product of $\Pe$ with any ccc poset is also ccc; it is \emph{Knaster}  (or has \emph{property-$\mathcal{K}$}) if every uncountable subset of $\Pe$ contains an uncountable subset consisting of pairwise compatible elements. More generally, for $k\geq 2$, $\Pe$ is \emph{$k$-Knaster} if every uncountable subset of $\Pe$ contains an uncountable subset such that any $k$-many of its elements have a common lower bound. Thus, Knaster is the same as $2$-Knaster. $\Pe$  has \emph{precalibre-$\aleph_1$} if every uncountable subset of $\Pe$ has an uncountable subset such that any finite set of its elements has a common lower bound;   it is \emph{$\sigma$-linked} (or \emph{$2$-linked}) if it can be partitioned into countably-many pieces so that each piece is pairwise compatible. More generally, for $k\geq 2$, $\Pe$ is $k$-linked if it can be partitioned into countably-many pieces so that any $k$-many elements in the same piece have a common lower bound. Finally, $\Pe$ is \emph{$\sigma$-centered} if it can be partitioned into countably-many pieces so that any finite number of elements in the same piece have a common lower bound. We have the following implications, for every $k\geq 2$:
$$\mbox{$\sigma$-centered}\Rightarrow\!\mbox{$k$-linked}\Rightarrow\!k\mbox{-Knaster}\Rightarrow\!\mbox{productive-ccc}\Rightarrow \!\mbox{ccc},$$
and 
$$\mbox{$\sigma$-centered}\Rightarrow \mbox{precalibre-$\aleph_1$}\Rightarrow k\mbox{-Knaster}.$$
These are the only implications that can be proved in ZFC.

\medskip

For a class of ccc posets satisfying some property $\Gamma$,  and an infinite cardinal $\kappa$, \emph{Martin's Axiom for $\Gamma$ and for families of $\kappa$-many dense open sets}, denoted by $MA_{\kappa}(\Gamma)$, asserts: for every $\Pe$ that satisfies the property $\Gamma$ and every family $\{ D_\alpha :\alpha <\kappa\}$ of dense open subsets of $\Pe$, there exists a filter $G\subseteq \Pe$ that is \emph{generic} for the family, that is, $G\cap D_\alpha \ne \emptyset$ for every $\alpha <\kappa$. 

 When $\kappa=\aleph_1$ we omit the subscript and write $MA(\Gamma)$ for $MA_{\aleph_1}(\Gamma)$. Also, for an infinite cardinal $\theta$, the notation $MA_{<\theta}(\Gamma)$ means: $MA_{\kappa}(\Gamma)$ for all $\kappa <\theta$. 
The axiom $MA_{\aleph_0}(\Gamma)$ is provable in ZFC; and it is consistent, modulo ZFC, that the Continuum Hypothesis fails and $MA_{< 2^{\aleph_0}}(\Gamma)$ holds (see \cite{K}, or \cite{Je}). \emph{Martin's axiom}, denoted by $MA$, is $MA(ccc)$.

Thus,  we have the following implications, for every $k\geq 2$:
$$MA_\kappa(ccc) \Rightarrow MA_{\kappa}(\mbox{productive-ccc})\Rightarrow$$
$$\Rightarrow  MA_{\kappa}(k\mbox{-Knaster})\Rightarrow
 MA_{\kappa}(\mbox{$k$-linked})\Rightarrow  MA_{\kappa}(\mbox{$\sigma$-centered}),$$
 and $$MA_{\kappa}(k\mbox{-Knaster})\Rightarrow MA_{\kappa}(\mbox{precalibre-$\aleph_1$}) \Rightarrow MA_{\kappa}(\mbox{$\sigma$-centered}).$$

For all the  facts mentioned in the rest of the paper without a proof, as well as for all undefined notions and notations, see \cite{Je}.

\section{The property $Pr_k$}

Let us consider the following property of partial orderings, weaker than the $k$-Knaster property.

\begin{definition}
\label{def1}
For $k\geq 2$, let $Pr_k(\Qu)$ mean that $\Qu$ is a forcing notion such that if $p_\varepsilon \in \Qu$, for all $\varepsilon <\aleph_1$, then we can find $\bar{u}$ such that:
\begin{enumerate}
\item[(a)] $\bar{u}=\langle u_\xi :\xi <\aleph_1\rangle$.
\item[(b)] $u_\xi$ is a finite subset of $\aleph_1$.
\item[(c)] $u_{\xi_0} \cap u_{\xi_1} =\varnothing$, whenever $\xi_0 \ne \xi_1$.
\item[(d)] If $\xi_0< \ldots <\xi_{k-1}$, then we can find $\varepsilon_l \in u_{\xi_l}$, for $l<k$, such that $\{ p_{\varepsilon_l} :l<k\}$ have a common lower bound.
\end{enumerate}
\end{definition}

Notice that  $Pr_k(\Qu)$ implies that $\Qu$ is ccc, and that $Pr_{k+1}(\Qu)$ implies $Pr_k(\Qu)$. Also note that if $\Qu$ is  $k$-Knaster, then $Pr_k(\Qu)$. For given a subset $\{ p_\varepsilon :\varepsilon <\aleph_1\}$ of $\Qu$, there exists an uncountable $X\subseteq \aleph_1$ such that $\{ p_{\varepsilon_l} :l<k\}$ has a common lower bound, for every $\varepsilon_0 <\ldots <\varepsilon_{k-1}$ in $X$, so we can take $u_\xi$ to be the singleton that contains the $\xi$-th element of $X$. Finally, observe that if $\Qu$ has precalibre-$\aleph_1$, then $Pr_k(\Qu)$ holds for every $k\geq 2$.

\medskip

Recall that if $T$ is an Aronszajn tree on $\omega_1$, then the forcing that specializes $T$ consists of finite functions $p$ from $\omega_1$ into $\omega$ such that if $\alpha\ne\beta$ are in the domain of $p$ and are comparable in the tree ordering, then $p(\alpha)\ne p(\beta)$. The ordering is the reversed inclusion. It is consistent, modulo ZFC, that the specializing forcing is not productive-ccc, an example being the case when $T$ is a Suslin tree. However, we have the following:

\begin{lemma}
\label{AtreePrk}
If $T$ is an Aronszajn tree and $\Qu=\Qu_T$ is the forcing that specializes $T$ with finite conditions, then  $Pr_k(\Qu)$ holds, for every $k\geq 2$.
\end{lemma}

\begin{proof}
Without loss of generality, $T=(\omega_1, <_T)$. Let $p_\alpha \in \Qu$, for $\alpha <\aleph_1$. 
By a $\Delta$-system argument we may assume that $\{ dom(p_\alpha):\alpha <\aleph_1\}$ forms a $\Delta$-system, with root $r$. Moreover, we may assume that for some fixed $n$, $|dom(p_\alpha)\setminus r|=n$, for all $\alpha <\omega_1$.
Let $\langle \alpha_1,\ldots ,\alpha_n\rangle$ be an enumeration of $dom(p_\alpha)\setminus r$. We may also assume that if $\alpha <\beta$, then the highest level of $T$ that contains some $\alpha_i$ ($1\leq i\leq n$) is strictly lower than the lowest level of $T$ that contains some $\beta_j$ ($1\leq j\leq n$).

Fix a uniform ultrafilter $D$ over $\omega_1$. For each $\alpha <\omega_1$ and 
$1\leq i,j\leq n$, let
$$D_{\alpha ,i,j}:=\{ \beta >\alpha : \alpha_i <_T \beta_j\}$$
and let
$$D_{\alpha , i, 0}:=\{ \beta >\alpha : \alpha_i \not <_T \beta_j, \mbox{ all }j\}.$$
For every $\alpha$ and every $i$, there exists $j_{\alpha ,i}\leq n$ such that $D_{\alpha , i,j_{\alpha ,i}}\in D$. Moreover,  for every $1\leq i\leq n$, there exists $E_i\in D$ such that $j_{\alpha ,i}$ is fixed, say with value $j_i$, for all $\alpha \in E_i$. 
We claim that $j_i=0$, for all $1\leq i\leq n$. For suppose $i$ is so that $j_i\ne 0$. Pick $\alpha <\beta <\gamma$ in $E_i\cap D_{\alpha ,i,j_i}\cap D_{\beta ,i,j_i}$. Then $\alpha_i , \beta_i <_T \gamma_{j_i}$, hence $\alpha_i <_T \beta_i$. This yields an $\omega_1$-chain in $T$, which is impossible.
Now let  $E:=\bigcap_{1\leq i\leq n}E_i \in D$.

We claim that for every $m$ and every $\alpha$ we can find $u\in [\omega_1 \setminus \alpha]^m$ such that if $\beta < \gamma $ are in $u$, then $\beta_i \not  <_T \gamma_j$, for every $1\leq i,j\leq n$. Indeed, given $m$ and $\alpha$, choose any $\beta^0 \in E\setminus \alpha$. Now given $\beta^0,\ldots ,\beta^l$, all in $E$, let $\beta^{l+1}\in E \cap \bigcap_{1\leq i\leq n}\bigcap_{l'\leq l}D_{\beta^{l'} ,i,0}$. Then  the set $u:=\{ \beta^0,\ldots ,\beta^{m-1}\}$ is as required.

We can now choose  $\langle u_\xi :\xi <\aleph_1\rangle$ pairwise-disjoint, with $| u_\alpha | >k\cdot n$,  so that if $\xi_1 <\xi_2$, then $sup(u_{\xi_1})<min(u_{\xi_2})$, and each $u_\xi$ is as above, i.e., if $\beta < \gamma $ are in $u_\xi$, then $\beta_i \not  <_T \gamma_j$, for every $1\leq i,j\leq n$. We claim that $\langle u_\xi :\xi <\aleph_1\rangle$ is as required.
So, suppose $\xi_0 <\ldots \xi_{k-1}$. We  choose $\alpha^\ell \in u_{\xi_\ell}$ by downward induction on $\ell\in \{ 0,\ldots ,k-1\}$ so that $\{ p_{\alpha^\ell}  : \ell <k \}$ has a common lower bound. Let $\alpha^{k-1}$ be any element of $u_{\xi_{k-1}}$. Now suppose $\alpha^{\ell +1},\ldots ,\alpha^{k-1}$ have been already chosen and we shall choose $\alpha^\ell$. We may assume that for each $\beta \in u_{\xi_\ell}$, $p_\beta$ is incompatible with  $p_{\alpha^{\ell'}}$,  some $\ell'\in \{ \ell +1, \ldots , k-1\}$, for otherwise we could take  as our $\alpha^\ell$ any $\beta \in u_{\xi_\ell}$ with $p_\beta$ compatible with all $p_{\alpha^{\ell'}}$,  $\ell'\in \{ \ell +1, \ldots , k-1\}$. Thus, for each $\beta \in u_{\xi_\ell}$ there exist $\ell'\in \{ \ell +1,\ldots ,k-1\}$   and $1\leq i,j\leq n$ such that $\beta_i <_T \alpha^{\ell'}_j$. So,  since $| u_{\xi_\beta}|>k\cdot n$, there must exist $\beta ,\beta'\in u_{\xi_\ell}$ and $\ell'$ such that $\beta_i ,\beta_{i'}<_T \alpha^{\ell'}_j$, for some $1\leq i,i',j\leq n$ with $\beta_i \ne \beta_{i'}$. But this implies that $\beta_i$ and $\beta_{i'}$ are $<_T$-comparable, contradicting our choice of $u_{\xi_\ell}$.
\end{proof}

We show next that the property $Pr_k$ for forcing notions is preserved under iterations with finite support, of any length.

\begin{lemma}
\label{iterations}
For any $k\geq 2$, the property $Pr_k$ is preserved under finite-support forcing iterations. That is, if
$$\langle \Pe_\alpha , \name{\Qu}_\beta ; \alpha \leq \lambda, \; \beta <\lambda \rangle$$
is a finite-support iteration of forcing notions such that $Pr_k(\Pe_0)$ and 
$\Vdash_{\Pe_\beta}``Pr_k(\name{\Qu}_\beta )"$, for every $\beta <\lambda$, then $Pr_k(\Pe_\lambda)$.
\end{lemma}

\begin{proof}
By induction on $\alpha \leq \lambda$. For $\alpha=0$ it is trivial. If $\alpha$ is a limit ordinal with  $cf(\alpha)\ne \aleph_1$, and $p_\varepsilon \in \Pe_\alpha$, for all $\varepsilon <\aleph_1$, then either uncountably many $p_\varepsilon$ have the same support (in the case $cf(\alpha) =\omega$) or the support of all $p_\varepsilon$ is bounded by some $\alpha'<\alpha$. In either case $Pr_k(\Pe_\alpha)$ follows easily  from the induction hypothesis.

If $cf(\alpha)=\aleph_1$, then we may use a $\Delta$-system argument, as in the usual proof of the preservation of the ccc.

So, suppose $\alpha =\beta +1$.  Let $p_\varepsilon\in \Pe_\alpha$, for all $\varepsilon <\aleph_1$. Without loss of generality, we may assume that $\beta \in dom(p_\varepsilon)$, for all $\varepsilon <\aleph_1$.

Since $\Pe_\beta$ is ccc, there is $q\in \Pe_\beta$ such that
$$q\Vdash_{\Pe_\beta}``|\{ \varepsilon : p_\varepsilon \restriction \beta \in \name{G}_\beta \}| =\aleph_1".$$

Let $G \subseteq \Pe_\beta$ be generic over $V$ and with $q\in G$. In $V[G]$ we have that $p_\varepsilon (\beta)[G]\in \name{\Qu}_\beta [G]$, and $Pr_k(\name{\Qu}_\beta [G])$ holds. 
So, there is $\langle u^0_\xi:\xi <\aleph_1\rangle$ as in Definition \ref{def1} for the sequence $\langle p_\varepsilon (\beta)[G]: p_\varepsilon \restriction \beta \in G\rangle$. So,
$$q\Vdash_{\Pe_\beta} ``\langle \name{u}^0_\xi :\xi <\aleph_1\rangle \mbox{ is as in  Definition \ref{def1} for }\langle p_\varepsilon (\beta) : p_\varepsilon \restriction \beta \in \name{G}_\beta \rangle".$$

For each $\xi$, let $(q_\xi ,u^1_\xi)$ be such that
\begin{enumerate}
\item[] $q_\xi \in \Pe_\beta$ and $q_\xi \leq q$.
\item[] $q_\xi \Vdash_{\Pe_\beta} `` \name{u}^0_\xi =u^1_\xi"$, so $u^1_\xi$ is finite.
\item[] $q_\xi \leq p_\varepsilon \restriction \beta$, for every $\varepsilon \in u^1_\xi$. (This can be ensured because if $\varepsilon \in u^1_\xi$, then $q_\xi \Vdash_{\Pe_\beta}``p_\varepsilon \restriction \beta \in \name{G}_\beta"$, so we may as well take  $q_\xi \leq p_\varepsilon \restriction \beta$.)
\end{enumerate}
Now apply the induction hypothesis for $\Pe_\beta$ and $\langle q_\xi :\xi <\aleph_1\rangle$ to obtain  $\langle u^2_\zeta :\zeta <\aleph_1\rangle$ as in the definition. We may assume, by refining the sequence if necessary, that $max (u^2_\zeta)<min(u^2_{\zeta'})$ whenever $\zeta <\zeta'$.

Let $u^\ast_\zeta := \bigcup \{ u^1_\xi :\xi \in u^2_\zeta\}$. We claim that $\bar{u}^\ast =\langle u^\ast_\zeta : \zeta <\aleph_1\rangle$ is as in the definition, for the sequence $\langle p_\varepsilon : \varepsilon < \aleph_1\rangle$. Clearly, the  $u^\ast_\zeta$ are finite and pairwise-disjoint. Moreover, given $\zeta_0<\ldots <\zeta_{k-1}$, we can find $\xi_0\in u^2_{\zeta_0}, \ldots , \xi_{k-1}\in u^2_{\zeta_{k-1}}$ such that in $\Pe_\beta$ there is a common lower bound $q_\ast$ to $\{ q_{\xi_0},\ldots ,q_{\xi_k}\}$.
Since $q_\ast \leq  q_{\xi_0},\ldots ,q_{\xi_{k-1}} \leq q$, there are some $q_{\ast\ast}\leq q_\ast$ and $\varepsilon_l\in u^1_{\xi_l}$, for each $l<k$, such that for some $\Pe_\beta$-name $\name{p}$,
$$q_{\ast\ast}  \Vdash ``\name{p}\leq_{\name{\Qu}_\beta}  p_{\varepsilon_0}(\beta),\ldots ,p_{\varepsilon_{k-1}}(\beta)".$$
Then the condition $q_{\ast\ast} \ast \name{p}$ is a common lower bound for the conditions $p_{\varepsilon_0},\ldots ,p_{\varepsilon_{k-1}}$.
\end{proof}

\section{On fragments of $MA$}

We shall now prove that $MA(Pr_{k+1})$ does not imply $MA(k\mbox{-linked})$, which yields a negative answer to the first question stated in the Introduction. The following is the main lemma.

\begin{lemma}
\label{lemmaPe_ast}
For $k\geq 2$, there is a forcing notion $\Pe_\ast =\Pe_\ast^k$ and  $\Pe_\ast$-names $\name{\Acal}$ and $\Qu_{\name{\Acal}}=\Qu^k_{\name{\Acal}}$ such that
\begin{enumerate}
\item $\Pe_\ast$ has precalibre-$\aleph_1$ and is of cardinality $\aleph_1$.
\smallskip

\item $\Vdash_{\Pe_\ast}``\name{\Acal}\subseteq [\aleph_1]^{k+1}"$
\smallskip

\item $\Vdash_{\Pe_\ast} ``\Qu_{\name{\Acal}}\! = \{ v\in [\aleph_1]^{<\aleph_0} \! : \! [v]^{k+1} \cap \, \name{\Acal}=\! \emptyset \} \mbox{, ordered by $\supseteq$}$, $\mbox{is $k$-linked}. \,"$
\item $\Vdash_{\Pe_\ast}``\name{I}_\alpha :=\{ v\in \Qu_{\name{\Acal}}: v\not \subseteq \alpha \}$ is dense, all $\alpha <\aleph_1$.''
\item $\Vdash_{\Pe_\ast} ``\mbox{If } v_\alpha \in \Qu_{\name{\Acal}}$ is such that $v_\alpha \not \subseteq \alpha$, $\mbox{ for }\alpha <\aleph_1$; and $u_\xi \in [\aleph_1]^{<\aleph_0}$, for $\xi <\aleph_1$, are non-empty and pairwise disjoint,  then there exist $\xi_0 <\ldots <\xi_{k}$ such that for every $\langle \alpha_\ell :\ell \leq k\rangle \in \prod_{\ell \leq k}u_{\xi_\ell}$ the set $\bigcup_{\ell \leq k}v_{\alpha_\ell}$ does not belong to $\Qu_{\name{\Acal}}.$"
\end{enumerate}
\end{lemma}

\begin{proof}
We define $\Pe_\ast$ by: $p\in \Pe_\ast$ if and only if $p$ has the form $(u,A,h)=(u_p, A_p, h_p)$, where
\begin{enumerate}
\item[$(a)$] $u\in [\aleph_1]^{<\aleph_0}$,
\item[$(b)$] 
$A\subseteq [u]^{k+1}$, and 
\item[$(c)$] $h:\wp_p \to \omega$, where $\wp_p :=\{ v\subseteq u :  [v]^{k+1}\cap A=\emptyset\}$, is such that if $w_0,\ldots ,w_{k-1}\in \wp_p$ and $h$ is constant on $\{ w_0,\ldots ,w_{k-1}\}$, then $w_0\cup \ldots \cup w_{k-1} \in \wp_p$.
\end{enumerate}
The order is given by: $p\leq q$ if and only if  $u_q\subseteq u_p$, $A_q =A_p \cap [u_q]^{k+1}$, and $h_q \subseteq h_p$ (hence $\wp_q=\wp_p \cap \mathcal{P}(u_q)$ and $h_p\restriction \wp_q =h_q$).

\medskip

($1$): Clearly, $\Pe_\ast$ has cardinality $\aleph_1$, so let us show that it has precalibre-$\aleph_1$. Given $\{ q_\xi=(u_\xi,A_\xi,h_\xi)  :\xi <\aleph_1\}\subseteq \Pe_\ast$ we can find an uncountable $W\subseteq \aleph_1$ such that:
\begin{itemize}
\item[(i)] The set $\{ u_{\xi}:\xi \in W\}$ forms a $\Delta$-system with heart $u_\ast$.
\item[(ii)] The sets $[u_\ast]^{k+1}\cap  A_{\xi}$, for $\xi \in W$, are all the same. Hence the sets $\wp_{\xi}\cap \mathcal{P}(u_\ast)$, for  $\xi \in W$, are also all the same.
\item[(iii)] The functions $h_{\xi}\restriction ( \wp_{\xi}\cap \mathcal{P}(u_\ast))$, for $\xi \in W$, are  all the same.
\item[(iv)] The  ranges of $h_{\xi}$, for $\xi \in W$, are all the same, say $R$. So, $R$ is finite.
\item[(v)] For each  $i\in R$,  the sets $\{ w\cap u_\ast: h_{\xi}(w)=i\}$, for $\xi \in W$, are the same.
\end{itemize}
We will show that every finite subset of $\{ q_\xi :\xi \in W\}$ has a common lower bound. Given $\xi_0,\ldots ,\xi_m \in W$, let $q=(u_q,A_q,h_q)$ be such that
\begin{itemize}
\item $u_q=\bigcup_{\ell \leq m}u_{\xi_\ell}$
\item $A_q=\bigcup_{\ell \leq m}A_{\xi_\ell}$. Note that this implies that the  $\wp_{\xi_\ell}$ are contained in $\wp_q=\{ v\subseteq u_q: [v]^{k+1}\cap A_q=\emptyset \} $. Indeed, if, say, $w\in \wp_{\xi_\ell}$, then $[w]^{k+1}\cap A_{\xi_\ell}=\emptyset$, and we claim that also $[w]^{k+1}\cap A_{\xi_j}=\emptyset$, for $j\leq m$. For if $v\in [w]^{k+1}\cap A_{\xi_j}$, with $j\ne \ell$, then $v\subseteq u_\ast$, and therefore $v\in [u_\ast]^{k+1}\cap A_{\xi_j}=[u_\ast]^{k+1}\cap A_{\xi_\ell}$. Hence, $v\in [w]^{k+1}\cap A_{\xi_\ell}$,  which is impossible because $[w]^{k+1}\cap A_{\xi_\ell}$ is empty. 
\item $h_q:\wp_q \to \omega$ is such that $h_q(v)=h_{\xi_\ell}(v)$ for all $v\in \wp_{\xi_\ell}$,  and the $h_q(v)$ are all distinct and greater than $sup\{ h_q(v): v\in \bigcup_{\ell \leq m}\wp_{\xi_\ell}\}$, for  $v\not \in \bigcup_{\ell \leq m}\wp_{\xi_\ell}$. Notice that $h_q$ is well-defined because  the restrictions $h_{\xi_\ell}\restriction ( \wp_{\xi_\ell}\cap \mathcal{P}(u_\ast))$, for $\ell \leq m$, are all the same.
\end{itemize}
We claim that $q\in \Pe_\ast$. For this, we only need to show that if $\{ w_0,\ldots ,w_{k-1}\}\subseteq \wp_q$ and $h_q$ is constant on $\{ w_0,\ldots ,w_{k-1}\}$, then $[\bigcup_{j <k}w_j ]^{k+1}\cap A_q=\emptyset$. So fix a set $\{ w_0,\ldots ,w_{k-1}\}\subseteq \wp_q$ and suppose $h_q$ is constant on it, say with constant value $i$. By definition of $h_q$ we must have $\{ w_0,\ldots ,w_{k-1}\}\subseteq \bigcup_{\ell\leq m} \wp_{\xi_\ell}$. Now suppose, towards a contradiction, that $v\in [\bigcup_{j<k}w_j]^{k+1}\cap A_{\xi_\ell}$, some $\ell \leq m$. Let $s=\{ w_j: j\leq m\}\cap \wp_{\xi_\ell}$, and let $t=\{ w_j:j\leq m\}\setminus s$. Thus, $v\subseteq \bigcup s \cup (\bigcup t \cap u_\ast)$.

By (v),  $$\{ w\cap u_\ast: h_{\xi_\ell}(w)=i\}=\{ w\cap u_\ast: h_{\xi_{\ell'}}(w)=i\}$$
for every $\ell'\leq m$.  
So, for every $w_j \in t$, there exists $w_j'\in \wp_{\xi_\ell}$ such that $w_j\cap u_\ast =w_j'\cap u_\ast$ and $h_{\xi_\ell} (w_j')=i$. Let $t'=s\cup \{ w_j': w_j\in t\}$.  Note that $t'\subseteq \wp_{\xi_\ell}$ and $t'\subseteq \{w:h_{\xi_\ell}(w)=i\}$. So,  
$$v\subseteq \bigcup t' \subseteq \bigcup \{ w:h_{\xi_\ell}(w)=i\}.$$ 
Thus, $v\in [\bigcup \{ w:h_{\xi_\ell}(w)=i\}]^{k+1}\cap A_{\xi_\ell}$. But this is impossible because $\bigcup \{ w:h_{\xi_\ell}(w)=i\}\in \wp_{\xi_\ell}$ and therefore 
$$[\bigcup \{ w:h_{\xi_\ell}(w)=i\}]^{k+1}\cap A_{\xi_\ell}=\emptyset.$$

\smallskip

Now one can easily check that $q\leq q_{\xi_0},\ldots ,q_{\xi_m}$. And this shows that the set $\{ q_\xi :\xi \in W\}$ is finite-wise compatible.

\medskip

($2$): Let  $$\name{\Acal}=\{ (\check{v},p): v \in A_p, \;p\in \Pe_\ast\}.$$
Thus, $\name{\Acal}$ is a name for the set $\bigcup\{A_p: p\in G\}$, where $G$ is the $\Pe_\ast$-generic filter. 
Clearly, (2) holds. 

\medskip

($3$): 
Let $$\Qu_{\name{\Acal}}=\{ (\check{v},p): v\in \wp_p, \, p\in \Pe_\ast\}.$$ Thus, $\Qu_{\name{\Acal}}$ is a name for the set $\bigcup\{ \wp_p :p\in G\}$, where $G$ is the $\Pe_\ast$-generic filter. Clearly, $\Vdash_{\Pe_\ast}``\Qu_{\name{\Acal}}=\{ v\in [\aleph_1]^{<\aleph_0} : [v]^{k+1} \cap \name{\Acal}=\emptyset \}"$. Moreover, if $G$ is $\Pe_\ast$-generic over $V$, then, by $(c)$, the function $\bigcup\{ h_p :p\in G\}$ witnesses that the interpretation $i_G(\Qu_{\name{\Acal}})$, ordered by $\supseteq$, is $k$-linked. 

\medskip

($4$): Clear.

\medskip

($5$): Suppose that $p\in \Pe_\ast$ forces  $\dot{v}_\alpha \in \Qu_{\name{\Acal}}$ is such that $\dot{v}_\alpha \not \subseteq \alpha$,  all $\alpha <\aleph_1$; and it also forces  $\dot{u}_\xi \in [\aleph_1]^{<\aleph_0},\mbox{ all } \xi <\aleph_1$, are non-empty and pairwise disjoint. 

For each $\xi<\aleph_1$, let $q_\xi =(u_\xi,A_\xi,h_\xi)  \leq p$ and let  $u_{\xi}^\ast\in [\aleph_1]^{<\aleph_0}$ and $\bar{v}_\xi^\ast =\langle v^\ast_{\xi, \alpha}:\alpha \in u^\ast_\xi\rangle$, with  $v^\ast_{\xi, \alpha}\in [\aleph_1]^{<\aleph_0}$,  be such that 
$$q_\xi\Vdash_{\Pe_\ast} ``\dot{u}_{\xi} =u^\ast_\xi\mbox{ and }
\dot{v}_\alpha =v^\ast_{\xi, \alpha}, \mbox{ for }\alpha \in u^\ast_\xi."$$
We may assume, by extending $q_\xi$ if necessary, that $u^\ast_\xi \cup  \bigcup_{\alpha \in u^\ast_{\xi}}v^\ast_{\xi ,\alpha}\subseteq u_\xi$.

As in ($1$),  we can find an uncountable $W\subseteq \aleph_1$ such that (i)-(v) hold for the set of conditions $\{ q_\xi :\xi \in W\}$. Hence  $\{ q_\xi :\xi \in W\}$ is pairwise compatible (in fact, finite-wise compatible), from which it follows that the set $\{ u^\ast_\xi :\xi \in W\}$ is pairwise disjoint.
Now choose $\xi_0 <\ldots <\xi_k$ from $W$ so that 
\begin{itemize}
\item The heart $u_\ast$ of the $\Delta$-system $\{ u_\xi :\xi \in W\}$ is an initial segment of $u_{\xi_\ell}$, all $\ell
\leq k$, 
\item $\sup(u_{\xi_\ell}) < \inf(u_{\xi_{\ell+1}}\setminus u_\ast)$, for all $\ell < k$, and
\item $u^\ast_{\xi_\ell}\subseteq (u_{\xi_\ell}\setminus u_\ast)$, for all $\ell \leq k$.
\end{itemize}

For each $\sigma=\langle \alpha_\ell :\ell \leq k\rangle \in \prod_{\ell \leq k}u^\ast_{\xi_\ell}$, pick $w_\sigma \in [\bigcup_{\ell \leq k}v^\ast_{\xi_\ell , \alpha_\ell}]^{k+1}$ such that $|w_\sigma \cap v^\ast_{\xi_\ell ,\alpha_\ell} \setminus \alpha_\ell |=1$, for all $\ell \leq k$. This is possible because $v^\ast_{\xi_\ell ,\alpha_\ell} \not \subseteq \alpha_\ell$. 

\begin{claim}
\label{claim1}
$w_\sigma \not \subseteq u_{\xi_\ell}$,  hence $w_\sigma \not \in A_{\xi_\ell}$, for all $\sigma \in \prod_{\ell \leq k}u^\ast_{\xi_\ell}$ and all $\ell \leq k$. 
\end{claim}

\begin{proof}[Proof of Claim] Fix $\sigma=\langle \alpha_\ell :\ell \leq k\rangle$ and $\ell \leq k$, and suppose, for a contradiction, that $w_\sigma \subseteq u_{\xi_\ell}$. Then $w_\sigma \subseteq (u_{\xi_\ell}\setminus u_\ast)$. If $\ell <k$, then since $\sup (u_{\xi_\ell})<\inf(u_{\xi_{\ell+1}}\setminus u_\ast) \leq  \inf (u^\ast_{\xi_{\ell +1}})\leq \alpha_{\ell +1}$,  we would have $w_\sigma \setminus \alpha_{\ell +1} =\emptyset$, which contradicts our choice of $w_\sigma$. But if $\ell =k$, then since  $\sup(v^\ast_{\xi_{\ell -1} , \alpha_{\ell -1}})\leq \sup(u_{\xi_{\ell -1}})<\inf (u_{\xi_\ell}\setminus u_\ast)$, we would have  $w_\sigma \cap v^\ast_{\xi_{\ell -1} , \alpha_{\ell -1}}=\emptyset$, which contradicts again our choice of $w_\sigma$.
\end{proof}

Now define $q=(u_q,A_q,h_q)$ as follows:
\begin{itemize}
\item $u_q=\bigcup_{\ell<k}u_{\xi_\ell}$
\item $A_q=(\bigcup_{\ell <k}A_{\xi_\ell})\cup \{ w_\sigma :\sigma \in  \prod_{\ell \leq k}u^\ast_{\xi_\ell} \}$. Note that since $w_\sigma \not \subseteq u_{\xi_\ell}$ (Claim \ref{claim1}), we have that $w_\sigma \not \in \wp_{\xi_\ell}$, for all $\sigma \in  \prod_{\ell \leq k}u^\ast_{\xi_\ell}$ and $\ell \leq k$. Hence, $\wp_{\xi_\ell}\subseteq \wp_q$, all $\ell \leq k$.
\item $h_q:\wp_q \to \omega$ is such that $h_q(v)=h_{\xi_\ell}(v)$ for  $v\in \wp_{\xi_\ell}$,  for all $\ell \leq k$, and the $h_q(v)$ are all distinct and greater than $sup\{ h_q(v): v\in \bigcup_{\ell \leq k}\wp_{\xi_\ell}\}$, for  $v\not \in\bigcup_{\ell \leq k}\wp_{\xi_\ell}$.
\end{itemize}
As in (1), we can now check that $q\in \Pe_\ast$. Moreover, by Claim \ref{claim1}, $A_{\xi_\ell}=A_q \cap [u_{\xi_\ell}]^{k+1}$. Hence, $q\leq q_{\xi_\ell}$, all $\ell \leq k$, and so 
$$q\Vdash_{\Pe_\ast} ``\dot{u}_{\xi_\ell} =u^\ast_{\xi_\ell}\mbox{ and }
\dot{v}_\alpha =v^\ast_{\xi_\ell, \alpha}, \mbox{ for }\alpha \in u^\ast_{\xi_\ell}."$$
And since $w_\sigma \in [\bigcup_{\ell \leq k}v^\ast_{\alpha_\ell}]^{k+1}\cap A_q$, for every $\sigma \in  \prod_{\ell \leq k}u^\ast_{\xi_\ell} $, we have that
$$q\Vdash_{\Pe_\ast}``\bigcup_{\ell \leq k}\dot{v}_{\alpha_\ell}\not \in \Qu_{\name{\Acal}}, \mbox{ for all } \langle \alpha_\ell :\ell \leq k\rangle \in \prod_{\ell \leq k}\dot{u}_{\xi_\ell}."$$

\end{proof}

\begin{lemma}
\label{lemmaPe_astPr}
Let $k\geq 2$ and let $\Pe_\ast$ be as in Lemma \ref{lemmaPe_ast}. Suppose $\name{\Qu}$ is a $\Pe_\ast$-name for a forcing notion that satisfies $Pr_{k+1}$. Then, 
$$\Vdash_{\Pe_\ast \ast \name{\Qu}}`` \mbox{There is no directed $G\subseteq \Qu_{\name{\Acal}}$ such that } \name{I}_\alpha \cap G \ne \emptyset,\mbox{ all }\alpha <\aleph_1."$$
where $\name{I}_\alpha$ is a name for the dense open set $\{ v\in \Qu_{\name{\Acal}} :  v\not \subseteq \alpha  \}$.
\end{lemma}

\begin{proof}
Suppose, for a contradiction, that $p\ast \dot{q}\in \Pe_\ast \ast \name{\Qu}$ and  
$$p\ast \dot{q}\Vdash_{\Pe_\ast \ast \name{\Qu}}`` \mbox{There exists } G\subseteq \Qu_{\name{\Acal}} \mbox{ directed, with  } \name{I}_\alpha \cap G \ne \emptyset,\mbox{ all }\alpha <\aleph_1."$$
Suppose $G_0\subseteq \Pe_\ast$ is a filter generic over $V$, with $p\in G_0$. So, in $V[G_0]$, letting $q=i_{G_0}(\dot{q})$ and $\Qu=i_{G_0}(\name{\Qu})$,  we have that for some $\Qu$-name  $\name{G}$, 
$$q\Vdash_{\Qu}`` \name{G}\subseteq \Qu_{\Acal} \mbox{ is directed and } I_\alpha \cap \name{G} \ne \emptyset,\mbox{ all }\alpha <\aleph_1."$$
For each $\alpha <\aleph_1$, let $q_\alpha \leq q$, and let $v_\alpha \in [\aleph_1]^{<\aleph_0}$ be such that 
$$q_\alpha \Vdash_{\Qu}``\check{v}_\alpha \in I_\alpha \cap \name{G}".$$
Thus, $v_\alpha \not \subseteq \alpha$, for all $\alpha <\aleph_1$.

Since $\Qu$ satisfies $Pr_{k+1}$, there exists $\bar{u}=\langle u_\xi :\xi <\aleph_1\rangle$ such that
\begin{enumerate}
\item[(a)] $u_\xi$ is a finite subset of $\aleph_1$, all $\xi <\aleph_1$, 
\item[(b)] $u_{\xi_0} \cap u_{\xi_1} =\varnothing$ whenever $\xi_0 \ne \xi_1$, and 
\item[(c)] if $\xi_0< \ldots <\xi_{k}$, then we can find $\alpha_\ell \in u_{\xi_\ell}$, for $\ell\leq k$, such that $\{ q_{\alpha_\ell} :\ell\leq k\}$ have a common lower bound.
\end{enumerate}
By Lemma \ref{lemmaPe_ast}, we can find $\xi_0< \ldots <\xi_{k}$ such that for every $\langle \alpha_\ell :\ell \leq k\rangle \in \prod_{\ell \leq k}u_{\xi_\ell}$ the set $\bigcup_{\ell \leq k}v_{\alpha_\ell}$ does not belong to $\Qu_{\Acal}.$

By $(c)$, let $\alpha_\ell \in u_{\xi_\ell}$, for $\ell\leq k$, be such that $\{ q_{\alpha_\ell} :\ell\leq k\}$ have a common lower bound, call it $r$. Then $r$ forces  that $\{ \check{v}_{\alpha_\ell} : \ell \leq  k\}\subseteq \name{G}$. And since $r$ forces that $\name{G}$ is directed, it also forces that $\bigcup_{\ell \leq k}v_{\alpha_\ell}\in \Qu_\Acal$. A contradiction.
\end{proof}

All elements are now in place to prove the main result of this section.

\begin{theorem}
Let $k\geq 2$. Assume $\lambda =\lambda^{<\theta}$, where $\theta =cf(\theta)>\aleph_1$.
Then there is a finite-support iteration
$$\bar{\Pe}=\langle \Pe_\alpha , \name{\Qu}_\beta ; \alpha \leq \lambda , \beta <\lambda \rangle$$
where 
\begin{enumerate}
\item $\Pe_0$ is the forcing $\Pe_\ast$ from Lemma \ref{lemmaPe_ast}.
\smallskip

\item $\Vdash_{\Pe_{\beta}}``Pr_{k+1}(\name{\Qu}_\beta)"$, for every $0<\beta <\lambda$.
\smallskip

\item In $V^{\Pe_\lambda}$ the axiom $MA_{<\theta}(Pr_{k+1})$ holds, hence in particular  (Lemma \ref{AtreePrk}) every Aronszajn tree on $\omega_1$  is special.
\smallskip

\item $\Qu_{\name{\Acal}}$ witnesses that $MA(k\mbox{-linked})$ fails in $V^{\Pe_\lambda}$.
\end{enumerate}
\end{theorem}

\begin{proof}
To obtain ($3$), we proceed in the standard way as in all iterations forcing (some fragment  of) $MA$, that is, we iterate all posets with the $Pr_{k+1}$ property and having cardinality $<\theta$, which are given by some fixed bookkeeping function (see \cite{Je} or \cite{K} for details). 

Since after forcing with $\Pe_0$ the rest of the iteration $\bar{\Pe}$ has the property $Pr_{k+1}$ (Lemma \ref{iterations}),  $(4)$ follows immediately from Lemma \ref{lemmaPe_astPr}.
\end{proof}

\begin{corollary}
For every $k\geq 2$, ZFC plus $MA(Pr_{k+1})$  does not imply $MA(k\mbox{-linked})$.
\end{corollary}

Thus, since $MA(Pr_{k+1})$ implies both $MA(\sigma\mbox{-centerd})$ and ``Every Aronszajn tree is special'', the  corollary answers in the negative the question from \cite{Ba}: Does $MA(\sigma$-centered$)$ plus ``Every Aronszajn tree is special" imply $MA(\sigma\mbox{-linked})$?

\section{On destroying precalibre-$\aleph_1$ while preserving the ccc}

We turn now to the second question stated in the Introduction (Steprans-Watson \cite{SW}): Is it consistent that there exists  a precalibre-$\aleph_1$ poset which is ccc but does not have precalibre-$\aleph_1$ in some forcing extension that preserves cardinals?

\medskip

Note that the forcing extension cannot be ccc, since  ccc forcing preserves the precalibre-$\aleph_1$ property. Also, as shown in \cite{SW}, assuming $MA$ plus the Covering Lemma, every forcing that preserves cardinals also preserves the precalibre-$\aleph_1$ property. Moreover, the examples provided in \cite{SW} of cardinal-preserving forcing notions that destroy the precalibre-$\aleph_1$ they do so by actually destroying the ccc property.

\medskip

A positive answer to Question 1 is provided by  the following theorem. But first, let us recall a strong form of Jensen's diamond principle,  \emph{diamond-star relativized to a stationary set $S$}, which is also due to Jensen. For $S$ a stationary subset of $\omega_1$, let
\begin{itemize}
\item[$\diamondsuit^\ast_{S}$:] There exists a sequence $\langle \mathcal{S}_\alpha :\alpha \in S\rangle$, where $\mathcal{S}_\alpha$ is a countable set of subsets of $\alpha$, such that for every $X\subseteq \omega_1$ there is a club $C\subseteq \omega_1$ with $X\cap \alpha \in \mathcal{S}_\alpha$, for every $\alpha \in C\cap S$.  
\end{itemize}

The principle $\diamondsuit^\ast_{S}$ holds in the constructible universe $L$, for every stationary $S\subseteq \omega_1$ (see \cite{De}, 3.5, for a proof in the case $S=\omega_1$, which can be easily adapted to any stationary $S$). Also, $\diamondsuit^\ast_{S}$ can be forced by a $\sigma$-closed forcing notion (see \cite{K}, Chapter VII, Exercises H18 and  H20, where it is shown how to force the even stronger form of diamond known as $\diamondsuit^+_S$).

\begin{theorem}
\label{theorem2}
It is consistent, modulo ZFC, that the CH holds and there exist
\begin{enumerate}
\item A forcing notion $T$ of cardinality $\aleph_1$ that preserves cardinals.
\item Two posets $\Pe_0$ and $\Pe_1$ of cardinality $\aleph_1$ that have precalibre-$\aleph_1$ and such that 
$$\Vdash_T``\Pe_0, \Pe_1 \mbox{ are ccc, but }\Pe_0 \times \Pe_1\mbox{ is not ccc}."$$
Hence $\Vdash_T``\Pe_0\mbox{ and }\Pe_1 \mbox{ don't have precalibre-$\aleph_1$}"$.
\end{enumerate}
\end{theorem}

\begin{proof}
Let $\{ S_1,S_2\}$ be a partition of $\Omega:=\{ \delta <\omega_1 :\delta\mbox{ a limit}\}$ into two stationary sets.
By a preliminary forcing, we may assume that $\diamondsuit^\ast_{S_1}$ holds. So, there exists $\langle \mathcal{S}_\alpha :\alpha \in S_1\rangle$, where $\mathcal{S}_\alpha$ is a countable set of subsets of $\alpha$, such that for every $X\subseteq \omega_1$ there is a club $C\subseteq \omega_1$ with $X\cap \alpha \in \mathcal{S}_\alpha$, for every $\alpha \in C\cap S_1$. In particular, the CH holds.   Using $\diamondsuit^\ast_{S_1}$,  we can  build an \emph{$S_1$-oracle}, i.e., an $\subset$-increasing sequence $\bar{M}=\langle M_\delta : \delta \in S_1\rangle$, with $M_\delta$ countable and transitive, $\delta \in M_\delta$, $M_\delta \models `` ZFC^- +\delta$ is countable",  and such that for every $A\subseteq \omega_1$ there is a club $C_A\subseteq \omega_1$ such that $A\cap \delta \in M_\delta$, for every $\delta \in C_A\cap S_1$. (For the latter, one simply needs to require that $\mathcal{S}_\delta \subseteq M_\delta$, for all $\delta \in S_1$.)
Moreover, we can build $\bar{M}$ so that it  has the following additional property:
\begin{itemize}
\item[($\ast$)]
For every regular uncountable cardinal $\chi$ and a well ordering $<^\ast_\chi$ of $H(\chi)$, the set of all (universes of) countable $N\preceq \langle H(\chi),\in ,<^\ast_\chi\rangle$ such that the Mostowski collapse of $N$ belongs to $M_{\delta}$, where $\delta:=N\cap \omega_1$, is stationary in $[H(\chi)]^{\aleph_0}$.
\end{itemize} 
The property ($\ast$) will be needed to prove that the tree partial ordering  $T$ (defined below) has many branches, and also to prove that the  product partial ordering $\Qu \times T$ (defined below) is $S_1$-proper (Claim \ref{claim10}), and so it does not collapse $\aleph_1$.

To ensure ($\ast$), take  a big-enough regular cardinal $\lambda$ and define  the sequence $\bar{M}$ so that, for every $\delta \in S_1$,    $M_\delta$  is the Mostowski  collapse of a countable  elementary substructure $X$ of $H(\lambda)$ that contains $\bar{M}\restriction \delta$,  all ordinals $\leq \delta$, and all elements of $\mathcal{S}_\delta$. 
To see that $(\ast)$ holds,  
fix a regular uncountable cardinal $\chi$, a well ordering $<^\ast_\chi$ of $H(\chi)$,  and a club $E\subseteq [H(\chi)]^{\aleph_0}$. Let $\bar{N}=\langle N_\alpha :\alpha <\aleph_1\rangle$ be an $\subset$-increasing and $\in$-increasing continuous chain of elementary substructures of $\langle H(\chi),\in,<^\ast_\chi\rangle$ with the universe of $N_\alpha$ in $E$, for all $\alpha <\aleph_1$. We shall find $\delta\in S_1$ such that  the transitive collapse of $N_\delta$ belongs to $M_\delta$, where $\delta=N\cap \omega_1$.

Fix  a bijection $h:\aleph_1\to \bigcup_{\alpha <\aleph_1}N_\alpha$, and let $\Gamma :\aleph_1 \times \aleph_1 \to \aleph_1$ be the standard pairing function (cf. \cite{Je}, 3). Observe  that the set
$$D:=\{ \delta <\aleph_1: \delta \mbox{ is closed under $\Gamma$ and $h$ maps $\delta$ onto }N_\delta\}$$
is a club. Now let
\begin{itemize}
\item[] $X_1:=\{ \Gamma (i,j):h(i)\in h(j)\}$
\item[] $X_2:=\{ \Gamma (\alpha ,i):h(i)\in N_\alpha\}$
\item[] $X_3:=\{ \Gamma (i,j): h(i) <^\ast_\chi h(j)\}$
\item[] $X:=\{ 3j+i: i\in \{ 1,2,3\}\}$
\end{itemize}
The set $S'_1:=\{ \delta \in S_1: X\cap \delta \in M_\delta\}$ is stationary. Thus, since the set  $C:=\{ \delta <\aleph_1 :\delta =N_\delta \cap \omega_1\}$ is a club, we can  pick $\delta \in C \cap D\cap S'_1$. Since $\delta \in D$, the structure
$$Y:=\langle X_2\cap \delta ,\{ \langle i,j\rangle: \Gamma (i,j)\in X_1 \cap \delta\}, \{ \langle i,j\rangle: \Gamma (i,j)\in X_3\cap \delta\} \rangle$$
is isomorphic to $N_\delta$, and therefore $Y$ and $N_\delta$ have the same transitive collapse. And since $\delta \in S'_1$,   $Y$ belongs to $M_\delta$. Hence, since $M_\delta \models ZFC^-$, the transitive collapse of $Y$ belongs to $M_\delta$. Finally, since $\delta \in C$, $\delta =N_\delta \cap \omega_1$.

\bigskip

We shall define now the forcing $T$.
Let us write $\aleph_1^{<\aleph_1}$ for  the set of all  countable sequences of countable ordinals. 
Let 
\begin{itemize}
\item[$T:=$] $\{ \eta \in \aleph_1^{<\aleph_1} :Range(\eta)\subset S_1$, $\eta$ is increasing and continuous,   of successor length, and  if $\varepsilon < lh(\eta)$, then $\eta \restriction \varepsilon \in M_{\eta (\varepsilon)} \}$. 
\end{itemize}
Let $\leq_T$ be the partial order on $T$ given by end-extension. 
Thus, $(T, \leq_T)$ is a tree. Note that, since $\delta \in M_\delta$ for every $\delta\in S_1$,  if $\eta \in T$, then $\eta \in M_{supRange(\eta)}$. Also notice that if $\eta \in T$, then $\eta^\frown\langle \delta \rangle \in T$, for every $\delta \in S_1$ greater than $supRange(\eta)$.   In particular, every node of $T$ of finite length has $\aleph_1$-many extensions of any bigger finite length. Now suppose $\alpha <\omega_1$ is a limit, and suppose, inductively,  that for every successor $\beta <\alpha$, every node  of $T$ of length $\beta$ has $\aleph_1$-many  extensions of every higher successor length below $\alpha$. We claim that every $\eta\in T$ of length less than $\alpha$ has $\aleph_1$-many extensions in $T$ of length $\alpha+1$. For every $\delta <\omega_1$, let $T_\delta:=\{ \eta \in T: supRange(\eta) <\delta\}$.   Notice that $T_\delta$ is countable: otherwise, uncountably-many $\eta \in T_\delta$ would have the same $supRange(\eta)$, and therefore they would all belong to the model $M_{supRange(\eta)}$, which is impossible because it is countable. 
Now fix a node $\eta \in T$ of length less than $\alpha$, and let 
 $B:=\{ b_\gamma :\gamma <\omega_1\}$ be an enumeration  of all the \emph{branches} (i.e., linearly-ordered subsets of $T$ closed under predecessors) $b$ of $T$ that contain $\eta$ and have length $\alpha$ (i.e., $\bigcup\{ dom(\eta'):\eta'\in b\}=\alpha$). We shall build a sequence $B^\ast:=\langle b^\ast_\xi:\xi <\omega_1\rangle$ of branches from $B$ so that the set $supB^\ast:=\langle supRange(\bigcup b^\ast_\xi):\xi <\omega_1\rangle$ is the increasing enumeration of a club.  To this end, start by fixing an increasing sequence $\langle \alpha_n:n<\omega \rangle$  of successor ordinals converging to $\alpha$,  with $\alpha_0 $ greater than the length of $\eta$. Then  let $b^\ast_0:=b_0$. Given $b^\ast_\xi$, let $\gamma$ be the least ordinal such that $\bigcup b_\gamma (\alpha_0)>supRange(\bigcup b^\ast_\xi)$, and let $b^\ast_{\xi +1}:=b_\gamma$. Finally,  given $b^\ast_\xi$ for all $\xi <\delta$, where $\delta<\omega_1$ is a limit ordinal, pick an  increasing sequence $\langle \xi_n:n<\omega\rangle$ converging to $\delta$. If $\delta\in S_1$, then since $M_\delta\models ``\delta$ is countable", we pick  $\langle \xi_n:n<\omega\rangle$ in $M_\delta$. By construction, the sequence $\langle supRange(\bigcup b^\ast_{\xi_n}):n<\omega\rangle$ is increasing. Now let $f:\alpha \to \aleph_1$ be such that $f \restriction [0,\alpha_0]=\bigcup b^\ast_{\xi_0}\restriction [0,\alpha_0]$, and $f \restriction (\alpha_n,\alpha_{n+1}]=\bigcup b^\ast_{\xi_{n+1}}\restriction (\alpha_n ,\alpha_{n+1}]$, for all $n<\omega$. Then set $b^\ast_\zeta:=\{ f\restriction \beta:\beta <\alpha \mbox{ is a successor}\}$. One can easily check that $b^\ast_\zeta$ is a branch of $T$ of length $\alpha$ with $supRange(\bigcup b^\ast_\zeta)=sup \{ supRange(\bigcup b^\ast_\xi):\xi <\zeta\}$.

By $(\ast)$ the set of all countable $N\preceq \langle H(\aleph_2),\in ,<^\ast_{\aleph_2}\rangle$ that contain $B^\ast$ and $\langle \alpha_n:n<\omega\rangle$, with $\alpha \subseteq N$, and such that the Mostowski collapse of $N$ belongs to $M_{\delta}$, where $\delta:=N\cap \omega_1$, is stationary in $[H(\chi)]^{\aleph_0}$. So, since the set $Lim(supB^\ast)$ of limit points of $supB^\ast$ is a club, there is such an $N$ with $\delta:=N\cap \omega_1 \in Lim(supB^\ast)$.  If $\bar{N}$ is the transitive collapse of $N$, we have that $B^\ast \restriction \delta \in \bar{N}\in M_{\delta}$, and so  in $M_{\delta}$ we can build, as above, the branch $b^\ast_\delta$. Therefore,  since $\delta=supRange(\bigcup b^\ast_\delta)$, we have that $\bigcup b^\ast_\delta \cup \{\langle \alpha , \delta\rangle\} \in T$ and extends $\eta$. We have thus shown that $\eta$ has $\aleph_1$-many extensions in $T$ of length $\alpha +1$. Even more, the set $\{ supRange(\bigcup b):b$ is a branch of  length $\alpha +1$ that extends $\eta\}$ is stationary.

Note however that since the complement of $S_1$ is stationary, $T$ has no  branch of length $\omega_1$, because the range of such a branch would be a club contained in $S_1$. But since every  $\eta \in T$ has extensions of length $\alpha +1$, for every $\alpha$ greater than or equal to the length of $\eta$, forcing with $(T, \geq_T)$ yields  a branch of $T$ of length $\omega_1$.

\bigskip

In order to obtain the forcing notions $\Pe_0$ and $\Pe_1$ claimed by the theorem, we need first to force with the forcing $\Qu$, which we define as follows. 
For $u$  a subset of $T$, let $[u]_T^2$ be the set of all pairs $\{ \eta ,\nu \} \subseteq u$ such that $\eta \ne \nu$ and $\eta$ and $\nu$ are $<_T$-comparable. Let  
\begin{itemize}
\item[$\Qu:=$]
$\{ p:[u]^2_T\to \{ 0,1\}: u\mbox{ is a finite subset of }T\}$,
\end{itemize}
ordered by reversed inclusion.

It is easily seen that $\Qu$ is ccc, and it has cardinality $\aleph_1$, so forcing with $\Qu$ does not collapse cardinals, does not change cofinalities,  and preserves cardinal arithmetic. (In fact, $\Qu$ is equivalent, as a forcing notion, to the poset for adding $\aleph_1$ Cohen reals, which is $\sigma$-centered, but we shall not make use of this fact.)

Notice that if $G\subseteq \Qu$ is a generic filter over $V$, then $\bigcup G :[T]_T^2\to \{ 0,1\}$.

\bigskip

Recall that, for $S\subseteq \aleph_1$ stationary,  a forcing notion $\Pe$ is called \emph{$S$-proper} if for all (some) large-enough regular cardinals $\chi$ and all (stationary-many) countable $\langle N,\in\rangle \preceq \langle H(\chi),\in\rangle$ that contain $\Pe$ and such that $N\cap \aleph_1 \in S$,  and all $p\in \Pe \cap N$, there is a condition $q\leq p$ that is $(N,\Pe)$-generic. If $\Pe$ is $S$-proper, then it does not collapse $\aleph_1$. (See \cite{Sh}, or \cite{Go} for details.)

\begin{claim}
\label{claim10}
The forcing $\Qu \times T$ is $S_1$-proper, hence it does not collapse $\aleph_1$.
\end{claim}
%

\begin{proof}[Proof of the claim]
Let $\chi$ be a large-enough regular cardinal, and let $<^\ast_\chi$ be a well-ordering of $H(\chi)$. Let $N\preceq \langle H(\chi ),\in, <^\ast_\chi \rangle$ be countable and such that $\Qu \times T$  belongs to $N$, $\delta:=N\cap \aleph_1 \in S_1$, and the Mostowski collapse of $N$ belongs to $M_\delta$.
 Fix $(q_0, \eta_0)\in (\Qu\times T)\cap N$. It will be sufficient to find a condition $\eta_\ast\in T$ such  that $\eta _0 \leq_T \eta_\ast$ and $(q_0,\eta_\ast)$ is $(N, \Qu\times T )$-generic.

Let
$$\Qu_\delta:=\{ p\in \Qu: \mbox{ if } \{ \eta ,\nu\}\in dom(p), \mbox{ then }\eta ,\nu \in T_\delta\}.$$
Thus, $\Qu_\delta$ is countable. Moreover, notice that $T_\delta =T\cap N$, and therefore  $\Qu_\delta = \Qu\cap N$. Hence, $T_\delta$ and $\Qu_\delta$ are the Mostowski collapses of $T$ and $\Qu$, respectively, and so they belong to $M_\delta$.

In $M_\delta$, let $\langle (p_n,D_n):n<\omega\rangle$ list all pairs $(p,D)$ such that $p\in \Qu_\delta$, and $D$ is a dense open subset of $\Qu_\delta \times T_\delta$ that belongs to the Mostowski collapse of $N$. That is, $D$ is the Mostowski collapse of a dense open subset of $\Qu\times T$ that belongs to $N$.

Also in $M_\delta$, fix an increasing sequence $\langle \delta_n :n<\omega \rangle$ converging to $\delta$, and let $$D'_n:=\{ (p, \nu)\in D_n: lh(\nu)>\delta_n\}.$$ 
Clearly, $D'_n$ is dense open.

Note that, as the Mostowski collapse of $N$ belongs to $M_\delta$, we have that $(<^\ast_\chi \restriction (\Qu_\delta \times T_\delta)=(<^\ast_\chi \restriction (\Qu \times T))\cap N\in M_\delta$.

Now, still in $M_\delta$, and starting with $(q_0,\eta_0)$, we inductively choose a sequence $\langle (q_n,\eta_n):n<\omega\rangle$, with $q_n\in \Qu_\delta$ and $\eta_n\in T_\delta$, and such that if $n=m+1$, then:
\begin{enumerate}
\item[(a)] $p_n \geq q_n$ and $\eta_m <_T \eta_n$.
\item[(b)] $(q_n, \eta_n)\in D'_n$.
\item[(c)] $(q_n,\eta_n)$ is the $<^\ast_\chi$-least such that (a) and (b) hold.
\end{enumerate}
Then,   $\eta_\ast := (\bigcup_n \eta_n)\cup \{ \langle \delta,\delta\rangle\} \in T$, and $\eta^\ast \in M_\delta$, hence $(q_0, \eta_\ast)\in \Qu\times T$. Clearly, $(q_0,\eta_\ast)\leq (q_0, \eta_0)$.  
So, we only need to check  that $(q_0,\eta_\ast)$ is $(N,\Qu\times T)$-generic.

Fix an open dense $E\subseteq \Qu\times T$ that belongs to $N$. We need to see that $E\cap N$ is predense below $(q_0,\eta_\ast)$. So, fix $(r,\nu)\leq (q_0,\eta_\ast)$. 
Since $\Qu$ is ccc, $q_0$ is $(N,\Qu)$-generic, so we can find $r'\in \{ p: (p,\eta)\in E,\mbox{ some }\eta\} \cap N$ that is compatible with $r$.  Let $n$ be such that $p_n=r'$ and $D_n$ is the Mostowski collapse of $E$. Then $(p_n,\eta_n)$ belongs to the transitive collapse of $E$, hence to $E\cap N$, and is compatible with $(r,\nu)$, as $(p_n,\eta_\ast)\leq (p_n,\eta_n)$.
\end{proof}

We thus conclude that if $G\subseteq \Qu$ is a filter generic over $V$, then in $V[G]$ the forcing $T$  does not collapse $\aleph_1$, and therefore, being of cardinality $\aleph_1$, it preserves cardinals, cofinalities, and the cardinal arithmetic.

\bigskip

We shall now define the $\Qu$-names for the forcing notions $\name{\Pe}_\ell$, for $\ell \in \{ 0,1\}$, as follows: in $V^{\Qu}$, let  $\name{b}=\bigcup \name{G}$, where $\name{G}$ is the standard $\Qu$-name for the $\Qu$-generic filter over $V$. Then let
\begin{itemize}
\item[$\name{\Pe}_\ell:=$]$\{ (w,c): w\subseteq T$ is finite, $c$ is a function from $w$ into $\omega$ such that if $\{ \eta ,\nu\} \in [w]^2_T$ and $\name{b}(\{\eta ,\nu\})=\ell$, then $c(\eta)\ne c(\nu)\}$.
\end{itemize}
A condition $(w,c)$ is stronger than a condition $(v,d)$ if and only if $w\supseteq v$ and $c\supseteq d$.
\medskip

We shall show that if $G$ is $\Qu$-generic over $V$, then in the extension $V[G]$, the partial orderings $\Pe_\ell =\name{\Pe}_\ell[G]$, for $\ell \in \{ 0,1\}$, and  $T$ are as required.

\begin{claim}
\label{claimprecal}
In $V[G]$, $\Pe_\ell$ has precalibre-$\aleph_1$.
\end{claim}

\begin{proof}[Proof of the claim]
Assume $p_\alpha =(w_\alpha ,c_\alpha)\in \Pe_\ell$, for $\alpha <\omega_1$. We shall find an uncountable $S\subseteq \aleph_1$ such that $\{ p_\alpha: \alpha \in S\}$ is finite-wise compatible. For each $\delta \in S_2$, let
$$s_\delta :=\{ \eta \!\restriction\! (\gamma +1): \eta \in w_\delta,\,\mbox{and $\gamma$ is maximal such that } \gamma <  lh(\eta)\,\wedge\, \eta (\gamma)<\delta\}.$$
As $\eta$ is an increasing and continuous sequence of ordinals from $S_1$, hence disjoint from $S_2$, the set $s_\delta$ is well-defined. Notice that $s_\delta$ is a  finite subset of  $T_\delta:=\{ \eta \in T: supRange(\eta) <\delta\}$, which is countable.

Let $s^1_\delta :=w_\delta \cap T_\delta.$ Note that $s^1_\delta \subseteq s_\delta$.


Let $f:S_2\to \omega_1$ be given by $f(\delta)=max\{ supRange(\eta):\eta \in s_\delta\}$. 
Thus, $f$ is regressive, hence constant on a stationary $S_3\subseteq S_2$. Let $\delta_0$ be the constant value of $f$ on $S_3$. Then, $s_\delta \subseteq T_{\delta_0}$, for every $\delta \in S_3$. So, since $T_{\delta_0}$ is countable, 
 there exist $S_4\subseteq S_3$ stationary and $s_\ast$ such that $s_\delta =s_\ast$, for every $\delta \in S_4$. Further, there is a stationary $S_5\subseteq S_4$ and $s^1_\ast$ and $c_\ast$ such that for all $\delta \in S_5$,
$$ s^1_\delta =s^1_\ast, \; c_\delta \restriction s^1_\ast =c_\ast, \mbox{ and }\forall \alpha <\delta ( w_\alpha \subseteq T_\delta).$$
Hence, if $\delta_1 <\delta_2$ are from $S_5$, then not only $w_{\delta_1}\cap w_{\delta_2}=s^1_\ast$, but also if $\eta_1\in w_{\delta_1} -s^1_\ast$ and $\eta_2\in w_{\delta_2} - s^1_\ast$, then $\eta_1$ and $\eta_2$ are $<_T$-incomparable: for suppose otherwise, say $\eta_1 <_T \eta_2$. If $\gamma +1=lh(\eta_1)$, then  $\eta_2\restriction (\gamma +1)=\eta_1 <_T \eta_2$, and $\eta_2 (\gamma)=\eta_1(\gamma)<\delta_2$, by choice of $S_5$. Hence, by  the definition of $s_{\delta_2}$, $\eta_2 \restriction (\gamma +1)=\eta_1$ is an initial segment of some member of $s_{\delta_2}=s_\ast$,  and so it  belongs to $T_{\delta_1}$, hence $\eta_1 \in s^1_\ast$, contradicting the assumption that $\eta_1\not \in s^1_\ast$.

So, $\{ p_\delta :\delta \in S_5\}$ is as required.
\end{proof}

It only remains to show that forcing with $T$ over $V[G]$ preserves the ccc-ness of $\Pe_0$ and $\Pe_1$, but makes their product not ccc.

\begin{claim}
If $G_T$ is $T$-generic over $V[G]$, then in the generic extension $V[G][G_T]$, the forcing $\Pe_\ell$ is ccc. 
\end{claim}

\begin{proof}[Proof of the claim] First notice that, by the Product Lemma (see \cite{Je}, 15.9), $G$ is $\Qu$-generic over $V[G_T]$, and $V[G][G_T]=V[G_T][G]$. Now suppose $\name{A}=\{ (\name{w}_\alpha , \name{c}_\alpha) :\alpha <\omega_1\}\in V[G_T]$  is a $\Qu$-name for an uncountable  subset of $\Pe_\ell$. For each $\alpha<\omega_1$, let $p_\alpha\in \Qu$ and  $(w_\alpha ,c_\alpha)$ be such that $p_\alpha \Vdash ``(\name{w}_\alpha , \name{c}_\alpha)=(w_\alpha ,c_\alpha)"$. Let $u_\alpha$ be such that $dom(p_\alpha)=[u_\alpha]^2_T$. By extending $p_\alpha$, if necessary, we may assume that $w_\alpha \subseteq u_\alpha$, for all $\alpha <\omega_1$. 
We shall find $\alpha \ne \beta$ and a condition $p$ that extends both $p_\alpha$ and $p_\beta$ and forces that  $(w_\alpha ,c_\alpha)$ and $(w_\beta ,c_\beta)$ are compatible. For this,  first extend $(w_\alpha , c_\alpha)$ to $(u_\alpha , d_\alpha)$ by letting $d_\alpha$ give different values in $\omega \setminus Range (c_\alpha)$ to all $\eta \in u_\alpha \setminus w_\alpha$. 
We may  assume that the set $\{ u_\alpha:\alpha <\omega_1\}$ forms a $\Delta$-system with root $r$. Moreover, we may assume that $p_\alpha$ restricted to $[r]^2_T$ is the same for all $\alpha <\omega_1$, and also that $d_\alpha$ restricted to $r$ is the same for all $\alpha <\omega_1$. Now pick $\alpha \ne \beta$ and let $p:[u_\alpha \cup u_\beta]^2_T\to \{ 0,1\}$ be such that $p\restriction [u_\alpha]^2_T =p_\alpha$, $p\restriction [u_\beta]^2_T=p_\beta$, and $p(\{ \eta ,\nu\})\ne \ell$, for all other pairs in $[u_\alpha \cup u_\beta]^2_T$. Then, $p$ extends both $p_\alpha$ and $p_\beta$, and forces that $(u_\alpha ,d_\alpha)$ and $(u_\beta ,d_\beta)$ are compatible, hence it forces that $(w_\alpha ,c_\alpha)$ and $(w_\beta ,c_\beta)$ are compatible.
\end{proof}

\bigskip

But in $V[G][G_T]$, the product $\Pe_0 \times \Pe_1$ is not ccc. For let $\eta^\ast =\bigcup G_T$.  For every $\alpha<\omega_1$, let $p^\ell_\alpha :=(\{ \eta^\ast \restriction (\alpha +1)\}, c^\ell_\alpha)\in \Pe_\ell$, where $c^\ell_\alpha (\eta^\ast \restriction (\alpha +1))=0$. Then the set 
$\{ (p^0_\alpha ,p^1_\alpha): \alpha<\omega_1\}$ is an uncountable antichain.

\end{proof}

\end{document}